\RequirePackage{amsthm}
\documentclass[sn-mathphys,Numbered]{sn-jnl}

\usepackage{graphicx}
\usepackage{multirow}
\usepackage{amsmath,amssymb,amsfonts}
\usepackage{mathrsfs}
\usepackage[title]{appendix}
\usepackage{xcolor}
\usepackage{textcomp}
\usepackage{manyfoot}
\usepackage{booktabs}
\usepackage{algorithm}
\usepackage{algorithmicx}
\usepackage{algpseudocode}
\usepackage{listings}

\theoremstyle{thmstyleone}
\newtheorem{theorem}{Theorem}
\newtheorem{proposition}[theorem]{Proposition}
\newtheorem{lemma}[theorem]{Lemma} 

\newtheorem{conjecture}[theorem]{Conjecture}

\theoremstyle{thmstyletwo}
\newtheorem{example}[theorem]{Example}
\newtheorem{remark}[theorem]{Remark}

\theoremstyle{thmstylethree}
\newtheorem{definition}[theorem]{Definition}

\numberwithin{equation}{section}
\numberwithin{theorem}{section}

\begin{document}

\title[Article Title]{Bounds and Limiting Minimizers for a Family of Interaction Energies}

\author*[1]{\fnm{Cameron} \sur{Davies}}\email{cameron.davies@mail.utoronto.ca}

\affil*[1]{\orgdiv{Department of Mathematics}, \orgname{University of Toronto}, \orgaddress{\street{40 St. George St.}, \city{Toronto}, \postcode{M5S 2E4}, \state{Ontario}, \country{Canada}}}

\abstract{We study a two parameter family of energy minimization problems for interaction energies $\mathcal{E}_{\alpha,\beta}$ with attractive-repulsive potential $W_{\alpha,\beta}$. We develop a concavity principle, which allows us to provide a lower bound on $\mathcal{E}_{\alpha,\beta}$ if there exist $\beta_0<\beta<\beta_1$ with minimizers of $\mathcal{E}_{\alpha,\beta_0}$ and $\mathcal{E}_{\alpha,\beta_1}$ known. 
In addition to this, we also derive new conclusions about the limiting behaviour of $\mathcal{E}_{\alpha,\beta}$ for $\beta\approx 2.$ Finally, we describe a method to show that, for certain values of $(\alpha,\beta),$ $\mathcal{E}_{\alpha,\beta}$ cannot be minimized by the uniform distribution over a top-dimensional regular unit simplex.  
Our results are made possible by two key factors -- recent progress in identifying minimizers of $\mathcal{E}_{\alpha,\beta}$ for a range of $\alpha$ and $\beta$, and an analysis of $\inf\mathcal{E}_{\alpha,\beta}$ as a function on parameter space.}


\maketitle

\thanks{
\em \copyright 2023 by the author.  The author is grateful to Robert McCann, Ryan Matzke, Almut Burchard, Doug Hardin, Ihsan Topaloglu, and Rupert Frank for stimulating interactions.
 The author acknowledges partial support of his research by a Canada Graduate Scholarship - Doctoral Natural Sciences and Engineering Research Council of Canada. The author has no conflicts of interest to report. 
}

\section{Introduction}\label{sec-intro}
The purpose of this paper is to study minimizers of certain interaction energy functionals $\mathcal{E}_W:\mathcal{P}(\mathbb{R}^n)\to \mathbb{R}$ of the form 
\begin{equation}\label{eqn - general ie}\mathcal{E}_{W}[\mu]=\frac{1}{2}\iint_{\mathbb{R}^n\times\mathbb{R}^n}W(x-y)d\mu(x)d\mu(y),\end{equation}
with interaction kernel $W:\mathbb{R}^n\times\mathbb{R}^n\to \mathbb{R}$. More concretely, we study the two parameter family of potentials \begin{equation}\label{eqn - potential}W_{\alpha,\beta}(x-y)=\frac{|x-y|^\alpha}{\alpha}-\frac{|x-y|^\beta}{\beta},\end{equation} for $\alpha>\beta>-n,$ and for such potentials we will use the notation $\mathcal{E}_{\alpha,\beta}:=\mathcal{E}_{W_{\alpha,\beta}}.$ 

Energies of the form \eqref{eqn - general ie} are closely related to the dynamics of particles whose pairwise interactions are governed by $W.$ More specifically, it is well-known that the $d_2$-gradient flow of $\mathcal{E}_W$ is the aggregation equation 
\begin{equation}\label{eqn - aggregation}\frac{\partial \mu}{\partial t}=\nabla\cdot (\mu\nabla W * \mu)\end{equation} 
(see, for example, \cite[Section 11.2.1]{AGS} or \cite[Section 8.4.2] {Santambrogio}). Potential functions in the form \eqref{eqn - potential} are called \textit{attractive-repulsive power-law potentials,} and have the benefit of providing a relatively simple, yet still intriguing, model of particle dynamics where particles which are far apart attract each other, whereas particles which are close together repel each other. We will review the literature on the interaction energy \eqref{eqn - general ie} with potentials of the form \eqref{eqn - potential} later in the introduction. For the time being, we note that attractive-repulsive particle models have been an attractive choice across many disciplines, from collective animal behaviour, to granular media in physics, to game theory, which I describe in more detail elsewhere in \cite[Section 1.3]{MScThesis}.

We will focus on models of this behaviour which are based on isotropic power law potentials. However, we also remark on a few other similar problems which have received a lot of research interest. In particular, the case of anisotropic potentials has been particularly popular, and I refer the reader to the work of Carrillo and Shu in \cite{CSAnisotropic} for a review, and some interesting novel results. Another interesting model of attractive-repulsive dynamics arises in the Riesz kernel case, where the kernel is taken to be purely repulsive, but where our goal is to minimize the energy among probability measures supported on a specified compact set (often a sphere). We refer the reader to Borodachov, Hardin, and Saff's book \cite{BHS} for much of the theory of Riesz potentials. In particular, \cite[Chapter 2]{BHS} introduces the fundamentals of the theory. Finally, we should also mention that, in the case of attractive-repulsive potentials, we can restrict our focus to probability measures whose density is the indicator function of some set, leading to the minimizing sets problem studied by Carazzato, Pratelli, and Topaloglu in \cite{CPT}.

\subsection{Literature Review}
We now provide the reader with some background on the study of the aggregation equation with power-law potentials. A seminal contribution came from Balagu\'e, Carrillo, Laurent, and Raoul, who in \cite{BCLR} discovered a fundamental distinction between the cases $\beta>2$ and $\beta<2.$ Namely, in the `strongly repulsive' $\beta<2$ case, nearby particles repel each other strongly enough to prevent global minimizers of \eqref{eqn - general ie} from concentrating on sets of zero Hausdorff dimension, whereas in the `mildly repulsive' $\beta>2$ case, repulsion between nearby particles is so mild that it is easily overcome by attraction, leading minimizers of \eqref{eqn - general ie} to concentrate on sets of Hausdorff dimension $0$. It was later shown by Carrillo, Figalli, and Patacchini in \cite{CFP} that global minimizers for $\mathcal{E}_{\alpha,\beta}$ for $\alpha>\beta> 2$ must concentrate on subsets of $\mathbb{R}^n$ with finite cardinality. It should be noted that neither \cite{BCLR} nor \cite{CFP} address the case of the centripetal line $\beta=2$. Here, as we will discuss later in the introduction, the structure of global minimizers of \eqref{eqn - general ie} depends strongly on the value of $\alpha.$

In parallel to this study of the properties of abstract global minimizers, various researchers studied both the linear and non-linear stability of various steady states of the aggregation equation \eqref{eqn - aggregation}, work which would later inspire myself and others to explicitly identify global minimizers of \eqref{eqn - general ie} for certain values of $\alpha$ and $\beta$. 
Most relevant to this research, Sun, Uminsky, and Bertozzi investigated the stability of uniform distributions over spherical shells and found that, for regimes in parameter space where uniform distributions over spherical shells fail to be stable, the uniform distribution over the vertices of a top-dimensional unit simplex appeared to be an attractor for the dynamics in \cite{SUB}. Kolokolnikov and Von Brecht also joined this trio to investigate the stability of uniform distributions over `particle rings', or equally spaced points on circles in $\mathbb{R}^2$ in \cite{BKSUV}. Other notable contributions came from Albi, Balagu\'e, Carrillo, and Von Brecht, who studied particle rings in \cite{ABCV}, Balagu\'e, Carrillo, Laurent, and Raoul, who studied the radial stability of uniform distributions on spherical shells in \cite{BCLR2}, and Simione, who studied the nonlinear stability of simplex configurations in his thesis \cite{Simione}.

Spurred on by both research directions, myself and a number of other researchers began identifying explicit global minimizers. Since, for potentials of the form $W_{\alpha,\beta},$ the interaction energy \eqref{eqn - general ie} is invariant under translations and rotations of $\mu$, when we say that minimizers of $\mathcal{E}_{\alpha,\beta}$ are unique, we implicitly mean that they are unique up to translation and rotation. In general, the following three types of measures will be especially relevant to our work: 

\begin{definition}[Three Useful Measures]\label{def - measures}
    Let $n\ge 1$ and $r\ge 0.$ We define three useful classes of measures as follows:
    \begin{itemize}
        \item Let $\{x_0,\dots, x_n\}$ be points in $\mathbb{R}^n$ satisfying $|x_i-x_j|=\delta_{ij}$ (i.e. so that they form the vertices of a unit simplex) such that, for each $i\in \{0,\dots,n\},$ $|x_i|=\sqrt{\frac{n}{2n+2}}$, the last $n-i-1$ components of each $x_i$ are $0,$ and the $i+1$th component of each $x_i$ is positive. Then we call the measure $\mu_n^\Delta:=\frac{1}{n+1}\sum_{i=1}^n\delta_{x_i}$ the \textit{centred regular unit $n$-simplex measure}, or the \textit{unit simplex} for short.
        \item We define the \textit{cross-polytope measure of radius $r$} as $\mu^*_n(r):=\frac{1}{2n}\sum_{i=1}^n(\delta_{re_i}+\delta_{-re_i}).$
        \item We define the \textit{spherical shell measure of radius $r$} as the uniform distribution over $\partial B_r(0).$
    \end{itemize}
\end{definition}
\noindent The identification of global minimizers began with the work of Lim and McCann in \cite{LM}, who first showed that, for $\beta\ge 2,$ in the hard confinement limit $\alpha=\infty,$ the minimizers of the interaction $\mathcal{E}_{\infty,2}$ are uniformly distributed over the vertices of top-dimensional unit simplices (i.e. maximal collections of points at distance $1$ from each other). They then used $\Gamma$-convergence to show that the same result holds for $\mathcal{E}_{\alpha,\beta}$ with $\beta\ge 2$ and sufficiently large $\alpha$ in \cite[Theorem 1.3]{LM}. Building off of this work, and making use of the convexity properties of $\mathcal{E}_{4,2}$ discussed by Lopes in \cite{Lopes}, I joined Lim and McCann to explicitly identify and characterize minimizers of $\mathcal{E}_{\alpha,\beta}$ for a large portion of the mildly repulsive regime $\alpha>\beta>2$ and the centrifugal line $\alpha>\beta=2$ in \cite{DLMSimplex, DLMSphere}. In particular, because they will be relevant to the results of this paper, I summarize our contributions from \cite[Theorem 1.1]{DLMSimplex}, \cite[Theorem 1.5]{DLMSimplex}, \cite[Corollary 2.3]{DLMSphere}, and \cite[Theorem 2.4]{DLMSphere} in the following result:
\begin{theorem}[Minimizers in the Mildly Repulsive Regime and on the Centrifugal Line from \cite{DLMSimplex,DLMSphere}] Let $\alpha>\beta\ge 2.$ Then:
\begin{itemize}
    \item If $(\beta, n)=(2, 1)$ and $\alpha\ge 3,$ then \eqref{eqn - general ie} is uniquely minimized by the unit simplex. Moreover, there exists $\beta_1\ge 2$ such that $\mathcal{E}_{\alpha,\beta}$ is uniquely minimized by the unit simplex for any $\alpha>\beta\ge \beta_1$, and such that there exists a continuous, strictly decreasing threshold function $\alpha_{\Delta^1}:[2,\beta_n]\to[\beta_n,3]$ with the property that $\mathcal{E}_{\alpha,\beta}$ is uniquely minimized by the unit simplex for $\alpha>\alpha_{\Delta^1}(\beta),$ and not minimized by the unit simplex for $\alpha<\alpha_{\Delta^1}(\beta).$
    \item If $n\ge 2$ and $4>\alpha>\beta=2,$ then $\mathcal{E}_{\alpha,\beta}$ is minimized by a spherical shell of appropriate radius. 
    \item If $n\ge 2$ and $(\alpha, \beta)=(4,2),$ then $\mathcal{E}_{\alpha,\beta}$ is minimized by any probability measure supported on a sphere of radius $\sqrt{\frac{n}{2n+2}}$ which satisfies the second moment constraint $\int x\otimes x d\mu(x)=\frac{1}{2n+2}\operatorname{Id}.$
    \item If $n\ge 2,$ there exists $\beta_n\ge 2$ such that $\mathcal{E}_{\alpha,\beta}$ is uniquely minimized by the unit simplex for any $\alpha>\beta\ge \beta_n$, and such that there exists a continuous, strictly decreasing threshold function $\alpha_{\Delta^n}:[2,\beta_n]\to [\beta_n, 4]$ with the property that $\mathcal{E}_{\alpha,\beta}$ is uniquely minimized by the unit simplex for $\alpha>\alpha_{\Delta^n}(\beta),$ and not minimized by the unit simplex for $\alpha_{\Delta^n}(\beta).$
\end{itemize}
\end{theorem}

A substantial amount of work has also been undertaken to identify energy minimizers in the strongly repulsive regime, and along the part of the centrifugal line which was not addressed by myself, Lim, and McCann. This work began with Carrillo and Shu, who in \cite[Theorem 5.1]{CS} and \cite[Remark 5.8]{CS} used convexity to show that suitable densities on balls globally minimize $\mathcal{E}_{2,\beta}$ and $\mathcal{E}_{4,\beta}$ for certain values of $\beta<2.$ In addition to this, Frank completed the classification of minimizers on the centripetal line in \cite{Frank}, where he used an elegant argument based  on Fourier analysis to show that, for $n=1$ and $3>\alpha>2,$ $\mathcal{E}_{\alpha,2}$ is uniquely minimized by a certain density on an interval. Most recently, Frank and Matzke used convexity arguments and a careful analysis of hypergeometric functions to show that, for a wide range of parameters $\alpha\in [2,4]$ and for certain $\beta\le 2,$ minimizers of $\mathcal{E}_{\alpha,\beta}$ are spherical shells of an appropriate radius depending on $\alpha$ and $\beta$. We provide a more precise summary of these results (\cite[Theorem 5.1]{CS}, \cite[Remark 5.8]{CS}, \cite[Theorem 1]{Frank}, and \cite[Theorem 1]{FM}) below:
\begin{theorem}[Minimizers in the Strongly Repulsive Regime and on the Centrifugal line from \cite{CS,Frank,FM}] The following are true:
    \begin{itemize}
        \item Let $\alpha=2$ and $2-n< \beta<\min(4-n,2),$ or $\alpha=4$ and $2-n<\beta<\frac{2+2d-d^2}{d+1}.$ Then $\mathcal{E}_{\alpha,\beta}$ is uniquely minimized by a density on a ball of appropriate radius, where the density has formula given by \cite[Equation (5.9)]{CS} if $\alpha=2$, or \cite[Equation (5.10)]{CS} if $\alpha=4.$
        \item Let $(\beta,n)=(2,1),$ and let $\alpha\in (2,3).$ Then $\mathcal{E}_{\alpha,2}$ is uniquely minimized by a density on an interval, which are both explicitly stated in \cite[Theorem 1]{Frank}.
        \item Let $d\ge 2,$ $2\le\alpha\le 4,$ and $\frac{-10+3\alpha+7d-\alpha d-d^2}{d+\alpha-3}\le \beta\le 2,$ and assume that $(\alpha,\beta)\ne (4,2).$ Then $\mathcal{E}_{\alpha,\beta}$ is uniquely minimized by the uniform distribution over the sphere of radius $$R_{\alpha,\beta}=\left(\frac{\Gamma(\frac{d+\beta-1}{2})\Gamma(\frac{2d+\alpha-2}{2})}{\Gamma(\frac{d+\alpha-1}{2})\frac{2d+\beta-2}{2}}\right)^{\frac{1}{\alpha-\beta}}.$$
    \end{itemize}
\end{theorem}

\subsection{Contributions of this Paper}
We now explain the key contributions of the present work. We note that, at a fundamental level, most of the new results of this paper stem from past work on identifying global minimizers of $\mathcal{E}_{\alpha,\beta}$ for various values of $(\alpha,\beta),$ along with a careful study of the minimum value of $\mathcal{E}_{\alpha,\beta}$ as a function on the space of parameters $\alpha>\beta>-n$. More precisely, we define the following class of functions:
\begin{definition}[Minimal Energy on a Set of Probability Measures]
    Let $\mathcal{F}\subseteq \mathcal{P}(\mathbb{R}^n)$ be a set of probability measures. We define the \textit{minimal energy of $\mathcal{E}_{\alpha,\beta}$ on $\mathcal{F}$} by $$E_\alpha^{\mathcal{F}}(\beta):=\inf_{\mu\in\mathcal{F}}\mathcal{E}_{\alpha,\beta}[\mu].$$
    In the case where $\mathcal{F}=\mathcal{P}(\mathbb{R}^n),$ we denote $E_\alpha(\beta):=E^{\mathcal{P}(\mathcal{R}^n)}_\alpha(\beta).$
\end{definition}
\noindent 

We dedicate Section \ref{sec-convexity} to studying the minimal energy as a function of $(\alpha,\beta),$ and exploring the lower bounds on the energy which we can derive from this perspective. In particular, for $\alpha>0$, this function is concave in $\beta$ and, if we minimize over all of $\mathcal{P}(\mathbb{R}^n),$ then the minimum is also non-decreasing in $\beta:$
\begin{proposition}[Properties of the Minimal Energy]\label{prop - concave increasing}
    Let $\alpha>0$ be fixed, let $\beta\in (0,\alpha),$ and let $\mathcal{F}\subseteq \mathcal{P}(\mathbb{R}^n)$ be a non-empty family of probability measures. Then $E^\mathcal{F}_\alpha(\beta)\in \mathbb{R}$ for every $\beta\in (0,\alpha]$ and, in addition, $E_\alpha^\mathcal{F}:(0,\alpha]\to \mathbb{R}$ is a continuous, concave function. Moreover, $E_\alpha$ is non-decreasing on $(0,\alpha].$
\end{proposition}
In fact, if we fix $0<\beta_0<\beta_1<\alpha,$ then we can say that, for $\beta\in [\beta_0,\beta_1],$ $E_\alpha$ is strongly concave on $[\beta_0,\beta_1]$, according to the following definition which can, for example, be found in \cite[Section 2.1.3]{Nesterov}:
\begin{definition}[Strong Concavity]\label{def - strong concave}
    We say that $f:\mathbb{R}\to\mathbb{R}$ is strongly concave on $[a,b]$ with concavity parameter $m$ if, for all $x_0,x_1\in [a,b],$ 
    $$f((1-t)x_0+tx_1)\ge (1-t)f(x_0)+tf(x_1)+t(1-t)\frac{m}{2}(x_0-x_1)^2.$$
    If $f$ is $C^2,$ then it is strongly concave on $[a,b]$ if and only if $f''(x)\ge m$ for all $x\in [a,b].$
\end{definition}
If, moreover, we know the minimum value of $\mathcal{E}_{\alpha,\beta_0}$ and $\mathcal{E}_{\alpha,\beta_1},$ then we can deduce the following lower bound on $\mathcal{E}_{\alpha,\beta}$ for $\beta\in (\beta_0,\beta_1)$ by concave interpolation:
\begin{proposition}[Strong Concave Interpolation Bound on $\mathcal{E}_{\alpha,\beta}$]\label{prop - sc bound}
    Let $0<\beta_0<\beta_1\le  \alpha.$ Then $E_{\alpha}$ is strongly concave on $[\beta_0,\beta_1],$ with concavity parameter $\frac{1}{2\beta_1^2}\frac{n}{n+1}\left[\frac{1}{\alpha}-\frac{1}{\beta_1}\right].$ In particular, if $\beta_t=(1-t)\beta_0+t\beta_1,$ then 
    \begin{equation}\label{eqn - strong convexity estimate}E_\alpha(\beta_t)\ge (1-t)E_\alpha(\beta_0)+tE_\alpha(\beta_1)+\frac{t(1-t)}{4\beta_1^2}\frac{n}{n+1}\left[\frac{1}{\alpha}-\frac{1}{\beta_1}\right](\beta_1-\beta_0)^2.\end{equation}
\end{proposition}

In Section \ref{sec-transition}, we then turn our attention to analyzing the minimizers of $\mathcal{E}_{\alpha,\beta}$ for $\beta\approx 2.$ When doing so, we recall that Carrillo, Figalli, and Patacchini showed in \cite{CFP} that, for $\beta>2,$ all minimizers of $\mathcal{E}_{\alpha,\beta}$ consist of finite sums of $\delta$-masses. With this in mind, we introduce the following notation for spaces of discrete measures with at most $k$ points in their support:
\begin{definition}[Discrete Measures]
    Given $k\in\mathbb{N}$ and $\Omega\subseteq\mathbb{R},$ we define the set of (at most) \textit{$k$-point measures} $\mathcal{D}_k(\Omega)\subseteq \mathcal{P}(\Omega)$ by
    $$\mathcal{D}_k(\Omega):=\{\mu\in\mathcal{P}(\Omega)\ |\ |\operatorname{spt}(\mu)|\le k\},$$
    where $|A|$ denotes the cardinality of $A.$
\end{definition}
By using Proposition \ref{prop - concave increasing} and considering these sets of discrete measures, we are able to show the following result, which states that, if $n=1$ and $\alpha\in (2,3),$ or if $n\ge 2$ and $\alpha\in (2,4),$ the number of points in the support of minimizers of $\mathcal{E}_{\alpha,\beta}$ necessarily tends to $\infty$ as $\beta \searrow 2:$
\begin{theorem}[Cardinality of Energy Minimizers for $\beta\approx 2$]\label{thm - card limit}
    For $n=1,$ fix $\alpha\in (2,3).$ For $n\ge 2,$ fix $\alpha\in (2,4).$ Then $$\lim_{\beta\to 2^+}\min\{|\operatorname{spt}(\mu)|\ |\ \mu\in\operatorname{argmin} \mathcal{E}_{\alpha,\beta}[\cdot]\}=+\infty.$$
\end{theorem}
We are also able to use these properties and techniques to say that, in a certain sense, minimizers of $\mathcal{E}_{\alpha,\beta}$ tend to minimizers of $\mathcal{E}_{\alpha,2}$ in Theorem \ref{thm - weak shells}, Proposition \ref{prop - weak 1d}, and Theorem \ref{thm - weak simplices}.

Finally, in Section \ref{sec-threshold}, I continue the work which Lim, McCann, and I started in \cite{DLMSimplex} on identifying lower bounds for the threshold function $\alpha_{\Delta^n},$ which is defined to be the value such that, for $\alpha<\alpha_{\Delta^n}(\beta),$ unit simplices do not minimize $\mathcal{E}_{\alpha,\beta}$ and for $\alpha>\alpha_{\Delta^n}(\beta),$ unit simplices uniquely minimize $\mathcal{E}_{\alpha,\beta}.$ In particular, by comparing the energy of the unit simplex to the energy of the optimal cross-polytope as defined in Definition \ref{def - measures}, I define a new family of lower bounds $\underline{\alpha}_{\Delta^n}^*$ on $\alpha_{\Delta^n}.$ In particular, these bounds arise from defining a family of unimodal functions $\varphi_n$ in Definition \ref{def - unimodal}. Once we have done so, and checked a number of properties, we arrive at the following result:
\begin{proposition}[A Lower Bound for the Transition Threshold]\label{prop - new threshold bound}
    Fix $\alpha>\beta\ge 2,$ and define $\underline{\alpha}_{\Delta^n}^*(\beta)$ as the largest solution of $\varphi_n(\alpha)=\varphi_n(\beta).$ Then $\underline{\alpha}_{\Delta^n}^*(\beta)\le \alpha_{\Delta^n}(\beta),$ where $\alpha_{\Delta^n}(\beta)$ is the simplex transition threshold defined in \cite[Theorem 1.5]{DLMSimplex}.
\end{proposition}
\noindent While we have not been able to find a rigorous proof that this bound is superior to those set out in my joint work with Lim and McCann in \cite{DLMSimplex}, I conclude Section \ref{sec-threshold}, and hence this paper, by providing some evidence which supports the conclusion that $\underline{\alpha}_{\Delta^n}^*$ is a better bound than $\alpha_{\Delta^n},$ at least in some cases.

\section{Parameter Space Concavity and Resulting Bounds}\label{sec-convexity}

We begin by establishing some basic concavity properties, starting with the following general proposition, a special case of which provides the basis for later analysis:

\begin{proposition}[Convexity and Concavity Properties of Parametrized Energies]\label{prop-general convex}
    Let $\Phi$ and $\Psi$ be convex parameter spaces and for $(\varphi,\psi)\in \Phi\times\Psi,$ define $W_{\varphi,\psi}(x,y)=w_\varphi(|x-y|)+w_\psi(|x-y|).$ Assume that for each $(\varphi,\psi)\in \Phi\times\Psi,$ the interaction energy $$\mathcal{E}_{\varphi,\psi}[\mu]:=\frac{1}{2}\iint_{\mathbb{R}^n\times\mathbb{R}^n}W_{\varphi,\psi}(x,y)d\mu(x)d\mu(y)$$
    is uniformly bounded below for $\mu\in\mathcal{F}\subseteq\mathcal{P}(\mathbb{R}^d),$ and define $E^\mathcal{F}_{\psi}:\Phi\to\mathbb{R}$ by $E^\mathcal{F}_\psi(\varphi):=\inf_{\mu\in\mathcal{F}}\mathcal{E}_{\varphi,\psi}[\mu].$ Then:
    \begin{enumerate}
        \renewcommand{\theenumi}{\alph{enumi})}
        \item if, for a fixed $\psi\in\Psi,$ the map $\varphi\mapsto w_\varphi(t)$ is concave for each $t\in [0,\infty),$ then for any $\mu\in\mathcal{F},$ the map $\varphi\mapsto \mathcal{E}_{\varphi, \psi}[\mu]$ is concave on $\Phi$. 
        \item if $\psi$ is fixed and the map $\varphi\mapsto \mathcal{E}_{\varphi, \psi}[\mu]$ is concave for each $\mu\in \mathcal{F}$, then the map $\varphi\mapsto E^\mathcal{F}_{\varphi,\psi}$ is concave on $\Phi$.
    \end{enumerate}
\end{proposition}
\begin{proof}
 For part a), we fix $t\in[0,1]$ and $\varphi_0,\varphi_1\in\Phi$ and compute
\begin{align*}
    \mathcal{E}_{(1-t)\varphi_0+t\varphi_1,\psi_*}[\mu]  
 & =\frac12\iint_{\mathbb{R}^n\times\mathbb{R}^n} w_{(1-t)\varphi_0+t\varphi_1}(|x-y|)+w_{\psi_*}(|x-y|)\ d\mu(x)d\mu(y)\\
 & \ge \frac12\iint_{\mathbb{R}^n\times\mathbb{R}^n}(1-t)w_{\varphi_0}(|x-y|)+tw_{\varphi_1}(|x-y|)+w_{\psi_*}(|x-y|)\ d\mu(x)d\mu(y)\\
 & = (1-t)\mathcal{E}_{\varphi_0,\psi_*}[\mu]+t\mathcal{E}_{\varphi_1,\psi_*},
\end{align*}
as desired. Part b) follows from noting that, in this case, $E^\mathcal{F}_{\varphi,\psi_*}=\inf_{\mu\in F}\mathcal{E}_\varphi,\psi_*[\mu]$ is an infimum of concave functions of $\varphi$ and applying, for example, \cite[Theorem 5.5]{Rockafellar}.
\end{proof}

This general proposition forms the basis of our proof of Proposition \ref{prop - concave increasing}:

\begin{proof}[Proof of Proposition \ref{prop - concave increasing}]
    First note that, for every $\beta\in(0,\alpha],$ $E_\alpha^\mathcal{F}(\beta)\ge E_\alpha(\beta)>-\infty$ by \cite[Theorem 2.3]{CFT}. We next show that $E_\alpha^\mathcal{F}$ is concave. This is an immediate consequence of \ref{prop-general convex} applied with $\Phi=(0,\alpha],$ $\Psi=\{\alpha\},$ $w_\psi(t)=w_\alpha(t)=\frac{t^\alpha}{\alpha},$ and, for $\beta\in \Phi,$ $w_\beta(t)=-\frac{t^\beta}{\beta}.$ In particular, the assumption that the energy is uniformly bounded below follows again from \cite[Theorem 2.3]{CFT}. Likewise, concavity of the map $\beta\mapsto w_\beta(t)$ for $t\ge 0$ follows from computing $$-\frac{d^2}{d\beta^2}\frac{t^\beta}{\beta}=\begin{cases}
        -\frac{t^\beta\left((\beta\log t)^2-2(\beta\log t)+2\right)}{\beta^3}, & t>0\\
        0, & t=0,
    \end{cases}$$
    using the fact $s^2-2s+2\ge 1$ for any $s\in\mathbb{R}$ to see that this second derivative is negative for $t>0.$ Continuity follows from $E_\alpha^\mathcal{F}$ being concave through, for example, \cite[Corollary 10.1.1]{Rockafellar}.

    To prove that $E_\alpha^\mathcal{F}$ is non-decreasing, first fix $0<\beta<\beta'<\alpha.$ We recall that, from \cite[Proposition 2.1]{DLMSphere}, any global minimizer $\mu$ of $\mathcal{E}_{\alpha,\beta}$ in $\mathcal{P}(\mathbb{R}^n)$ satisfies the diameter support bound $\operatorname{diam}(\operatorname{spt}(\mu))\le e^{1/\beta}$ and likewise any global minimizer $\mu$ of $\mathcal{E}_{\alpha,\beta'}$ satisfies $\operatorname{diam}(\operatorname{spt}(\mu))\le e^{1/\beta'}<e^{1/\beta}.$ As such, we may write \begin{equation} \label{eqn - diam bound}E_\alpha(\beta)=\inf\{\mathcal{E}_{\alpha,\beta}(\mu)\ |\ \operatorname{diam}(\operatorname{spt}(\mu))\le e^{1/\beta}\}.\end{equation}
    and 
    \begin{equation} \label{eqn - diam bound2}E_\alpha(\beta')=\inf\{\mathcal{E}_{\alpha,\beta'}(\mu)\ |\ \operatorname{diam}(\operatorname{spt}(\mu))\le e^{1/\beta}\}.\end{equation}
    Moreover, notice that $-\frac{d}{d\beta}\frac{t^\beta}{\beta}=\frac{t^\beta(1-\log(t^\beta))}{\beta^2}\ge 0$ for $t\in (0,e^{1/\beta}]$ and $\frac{d}{d\beta}\frac{t^\beta}{\beta}=0$ if $t=0.$ Thus, the function $-\frac{t^\beta}{\beta}$ is non-decreasing in $\beta$ for $t\in [0,e^{1/\beta}].$ As such, for any  $\mu$ with $\operatorname{diam}(\operatorname{spt}(\mu))\le e^{1/\beta},$ we find that:
    \begin{align*}
        \mathcal{E}_{\alpha,\beta}[\mu] & = \frac12
        \iint_{\mathbb{R}^n\times\mathbb{R}^n} \frac{|x-y|^\alpha}{\alpha}-\frac{|x-y|^\beta}{\beta}d\mu(x)d\mu(y)\\
        & \le \frac12\iint_{\mathbb{R}^n\times\mathbb{R}^n} \frac{|x-y|^\alpha}{\alpha}-\frac{|x-y|^{\beta'}}{\beta'}d\mu(x)d\mu(y)\\
        & =\mathcal{E}_{\alpha,\beta'}[\mu].
    \end{align*}
    Taking an infimum over all $\mu\in\mathcal{P}(\mathbb{R}^n)$ satisfying the diameter bound, we deduce that $E_\alpha(\beta)\le E_\alpha(\beta'),$ as desired. 
\end{proof}

\begin{remark}
    A similar result to Proposition \ref{prop - concave increasing} holds if we fix $\beta<0$ and consider the minimal energy as a function of $\alpha\in (\beta,0).$ However, in this regime, much less is known about minimizers of the interaction energy, so fewer conclusions can be drawn from this result. 
\end{remark}

An immediate application of the concavity discussed in Proposition \ref{prop - concave increasing} is that, if $\beta_0,\beta_1$ are such that $E_\alpha(\beta_0)$ and $E_\alpha(\beta_1)$ are known, we may use linear interpolation to find explicit lower bounds for $E_\alpha(\beta)$ for $\beta\in [\beta_0,\beta_1].$ More precisely, the following holds:

\begin{lemma}[Concave Interpolation Bound on the Energy]\label{lem - linear interpolation}
    Let $0<\beta_0<\beta_1\le \alpha,$ and define $\beta_t:=(1-t)\beta_0+t\beta_1.$ Then $$E_\alpha(\beta_t)\ge (1-t)E_\alpha(\beta_0)+tE_\alpha(\beta_1).$$
\end{lemma}
\begin{proof}
    This is a direct consequence of Proposition \ref{prop - concave increasing} and the definition of concavity. 
\end{proof}
Of course, it is more insightful to apply the preceding lemma to cases where minimizers are explicitly known for $\beta_0$ and $\beta_1,$ but not for $\beta\in(\beta_0,\beta_1).$ However, it is often the case that, even for values of $(\alpha,\beta)$ with $E_\alpha(\beta)$ is explicitly known, the formula for $E_{\alpha}(\beta)$ can be quite unwieldy (see, for example, \cite[Theorem 1]{Frank} and \cite[Theorem 1]{FM}). As such, we instead provide an example application of Lemma \ref{lem - linear interpolation} as a proof-of-concept, and then provide an overview of the relevant literature.

\begin{example}\label{ex - sample bound}
    Let $n\ge 2.$ Then for $\alpha\in (2,4),$ and $\beta\in (2,\alpha),$ 
    $$E_\alpha(\beta)\ge -\frac{\alpha-\beta}{\alpha-2}\frac{2^{d-3}}{\sqrt{\pi}}\frac{\Gamma(\frac{d}{2})\Gamma(\frac{d+\alpha-1}{2})}{\Gamma(\frac{2d+\alpha-2}{2})}\left(\frac{1}{2}-\frac{1}{\alpha}\right)\left(\frac{\Gamma(\frac{d+1}{2})\Gamma(\frac{2d+\alpha-2}{2})}{\Gamma(\frac{d+\alpha-1}{2})\Gamma(d)}\right)^{\frac{\alpha}{\alpha-\beta}}.$$
\end{example}
\begin{proof}
    Let $\beta_0=2$ and $\beta_1=\alpha,$ so that $\beta=\frac{\alpha-\beta}{\alpha-2}\cdot 2+ \frac{\beta-2}{\alpha-2}\cdot \alpha.$ By definition $\mathcal{E}_{\alpha,\alpha}[\mu]=0$ for any $\mu\in\mathcal{P}(\mathbb{R}^n)$ so that $E_\alpha(\alpha)=0.$ Moreover, by \cite[Theorem 1]{FM}, 
    $$E_\alpha(2)=-\frac{2^{d-3}}{\sqrt{\pi}}\frac{\Gamma(\frac{d}{2})\Gamma(\frac{d+\alpha-1}{2})}{\Gamma(\frac{2d+\alpha-2}{2})}\left(\frac{1}{2}-\frac{1}{\alpha}\right)\left(\frac{\Gamma(\frac{d+1}{2})\Gamma(\frac{2d+\alpha-2}{2})}{\Gamma(\frac{d+\alpha-1}{2})\Gamma(d)}\right)^{\frac{\alpha}{\alpha-\beta}}.$$
    Thus, by Lemma \ref{lem - linear interpolation}, we deduce the desired result. 
\end{proof}

\begin{remark}[Use Cases of Lemma \ref{lem - linear interpolation}]\label{rem - concavity cases}
The bounds provided by Lemma \ref{lem - linear interpolation} are useful in cases where there exist $\beta_0<\beta_1$ such that minimizers of $\mathcal{E}_{\alpha,\beta_0}$ and $\mathcal{E}_{\alpha,\beta_1}$ are known, but minimizers for $\mathcal{E}_{\alpha,\beta}$ are unknown for at least some $\beta\in (\beta_0,\beta_1).$ Given the current state of the literature, Lemma \ref{lem - linear interpolation} can be used to provide energy bounds in the following situations:
\begin{itemize}
    \item $n=1,$ $\alpha\in (2,3),$ $\beta\in [2,\alpha].$ In this case, the minimal energy $E_{\alpha}(2)$ is explicitly known due to \cite[Theorem 1]{Frank}. On the other hand, $E_\alpha(\alpha)$ is identically zero, which allows us to interpolate. However, by applying the work on transition threshold bounds in \cite[Section 4]{DLMSimplex}, we can get better bounds for $\alpha\in \left(\frac{1}{\log(3/2)}, 3\right).$ In particular, following \cite[Definition 4.1]{DLMSimplex}, for such an $\alpha,$ define $\beta_1(\alpha)$ as the smallest solution to $$\frac{(3/2)^\alpha}{\alpha}=\frac{(3/2)^\beta}{\beta}$$
    which, in the notation of that paper, ensures that $\alpha=\alpha^*_\infty(\beta_1(\alpha)).$ Thus, by \cite[Corollary 4.5]{DLMSimplex}, $\mathcal{E}_{\alpha,\beta_1(\alpha)}$ is uniquely minimized by unit simplices, so that 
    $$
    E_\alpha(\beta_1(\alpha))=\frac{1}{4}\left(\frac{1}{\alpha}-\frac{1}{\beta_1(\alpha)}\right)=\frac{1}{4\alpha}(1-(3/2)^{\alpha-\beta_1(\alpha)}).$$
    \item $n\ge 2,$ $\alpha\in (2,4),$ $\beta\in [2,\alpha].$ As we discussed in Example \ref{ex - sample bound}, minimizers of $\mathcal{E}_{\alpha,2}$ for this range of $\alpha$ were shown to be spherical shells in \cite{DLMSphere}, and $E_\alpha(2)$ was calculated in \cite[Theorem 1]{FM}. However, analogously to the one-dimensional case, it is possible to get better bounds if $\alpha\in \left(\frac{2}{\log 2}, 4\right).$ More specifically, in this case, we can define $\beta_1(\alpha)$ as the smallest solution to $2^{\alpha/2}/\alpha=2^{\beta/2}/\beta,$ which again allows us to apply \cite[Corollary 4.5]{DLMSimplex} to check that $\mathcal{E}_{\alpha,\beta_1(\alpha)}$ is uniquely minimized by unit simplices, and hence that $$
    E_\alpha(\beta_1(\alpha))=\frac{1}{2}\frac{n}{n+1}\left(\frac{1}{\alpha}-\frac{1}{\beta_1(\alpha)}\right)=\frac{1}{2\alpha}\frac{n}{n+1}(1-(3/2)^{\alpha-\beta_1(\alpha)}).$$
\end{itemize}
\end{remark}

We now turn our attention to improving Lemma \ref{lem - linear interpolation} by using strong concavity as in Definition \ref{def - strong concave}. 
Since it will be useful later, and since we lack a source for this, we prove the following useful property of strongly concave functions.
\begin{lemma}[Strong Concavity of the Pointwise Infimum]\label{lem - inf str cvx}
    Let $\{f_i\}_{i\in I}$ be a family of strongly concave functions on $[a,b]$ which are uniformly bounded below on $[a,b]$ and which have uniform concavity parameter $m.$ Then the pointwise infimum 
    $f:=\inf_{i\in I}f_i$ is also strongly concave with concavity parameter $m.$
\end{lemma}
\begin{proof}
    Fix $x_0,x_1\in[a,b]$ and explicitly compute:
    \begin{align*}
        f((1-t)x_0+tx_1) & =\inf_{i\in I}f_i((1-t)x_0+tx_1)\\
        & \ge \inf_{i\in I}\left[(1-t)f_i(x_0)+tf_i(x_1)+t(1-t)\frac{m}{2}(x_0-x_1)^2\right]\\
        & \ge (1-t)\inf_{i\in I}f_i(x_0)+t\inf_{i\in I}f_i(x_1)+t(1-t)\frac{m}{2}(x_0-x_1)^2\\
        & =(1-t)f(x_0)+tf(x_1)+t(1-t)\frac{m}{2}(x_0-x_1)^2,
    \end{align*}
    as desired. 
\end{proof}

In the following lemmas, we outline a strategy for proving that $E_\alpha(\beta)$ is strongly concave on certain intervals, note that we can use it, along with the same techniques described in Example \ref{ex - sample bound} and Remark \ref{rem - concavity cases}, to derive lower bounds on the energy. We caution the reader that we make a number of simplifying estimates which mean that the concavity parameter we eventually derive is suboptimal. Because of this, we phrase the following train of thought as a series of short lemmas, rather than a single, longer result:

\begin{lemma}[Second Derivative of the Energy]\label{lem - 2nd beta derivative}
Fix $\alpha>\beta>0$ and let $\mu\in\mathcal{P}_c(\mathbb{R}^n)$ be a fixed, compactly supported probability measure. Then 
$$2\frac{d^2}{d\beta^2}\mathcal{E}_{\alpha,\beta}[\mu]=-\iint_{\mathbb{R}^n\times\mathbb{R}^n}\frac{|x-y|^\beta((\log(|x-y|^\beta))^2-2\log(|x-y|^\beta)+2)}{\beta^3}d\mu(x)d\mu(y),$$
with the convention that $0((\log 0)^2-2\log(0)+2)=0.$ 
\end{lemma}
\begin{proof}
    This is more or less a direct application of the dominated convergence theorem. To do so, we recall that, as we calculated in the proof of Proposition \ref{prop - concave increasing}, 
    $$\frac{d}{d\beta}W_{\alpha,\beta}(x,y)=\begin{cases}
        0 & \text{ if }x=y\\
        \frac{|x-y|^\beta(1-\log(|x-y|^\beta))}{\beta^2} & \text{ if }x\ne y. 
    \end{cases}$$
    and $$\frac{d^2}{d\beta^2}W_{\alpha,\beta}(x,y)=\begin{cases}
        0 & \text{ if }x=y\\
        -\frac{|x-y|^\beta((\log(|x-y|^\beta))^2-2\log(|x-y|^\beta)+2)}{\beta^3} & \text{ if }x\ne y.\end{cases}.$$
    Each of these derivatives is continuous, and hence bounded on the compact set $\operatorname{spt}(\mu\otimes \mu),$ which allows us to apply the Lebesgue dominated convergence theorem as in \cite[Theorem 2.27(b)]{Folland} to recover the desired result. 
\end{proof}

Applying this calculation, along with some small estimates, yields the following differential inequality, which may be of interest in its own right:
\begin{lemma}[Differential Inequality for the Energy]\label{lem - diff ineq}
    Let $\alpha>\beta>0$ and let $\mu\in \mathcal{P}_c(\mathbb{R}^n).$ Then $$\frac{d^2}{d\beta^2}\mathcal{E}_{\alpha,\beta}[\mu]\le \frac{1}{\beta^2}\mathcal{E}_{\alpha,\beta}[\mu].$$
\end{lemma}
\begin{proof}
    By applying Lemma \ref{lem - 2nd beta derivative} and noticing that $s^2-2s+2\ge 1$ for all $s\in \mathbb{R},$ we deduce that 
    \begin{align*}
        \frac{d^2}{d\beta^2}\mathcal{E}_{\alpha,\beta}[\mu] & \le \frac{1}{2}\iint_{\mathbb{R}^n\times\mathbb{R}^n}-\frac{|x-y|^\beta}{\beta^3}d\mu(x)d\mu(y)\\
        & \le \frac{1}{2\beta^2}\iint_{\mathbb{R}^n\times\mathbb{R}^n}\frac{|x-y|^\alpha}{\alpha}-\frac{|x-y|^\beta}{\beta}d\mu(x)d\mu(y)\\
        & = \frac{1}{\beta^2}\mathcal{E}_{\alpha,\beta}(\mu).
    \end{align*}
\end{proof}

Finally, we can use the preceding differential inequality, along with a comparison to the simplex, to deduce the following uniform bound. 
\begin{lemma}[Concavity Parameter of the Energy]\label{lem - unif bd}
    Let $0<\beta_0<\beta<\beta_1\le \alpha$ and assume that $\mu\in\mathcal{P}_c(\mathbb{R}^n)$ satisfies $\mathcal{E}_{\alpha,\beta}[\mu]\le \frac{1}{2}\frac{n}{n+1}\left[\frac{1}{\alpha}-\frac{1}{\beta_1}\right].$ Then $$\frac{d^2}{d\beta^2}\mathcal{E}_{\alpha,\beta}[\mu]\le \frac{1}{2\beta_1^2}\frac{n}{n+1}\left[\frac{1}{\alpha}-\frac{1}{\beta_1}\right].$$
\end{lemma}
\begin{proof}
    By Lemma \ref{lem - diff ineq} and the assumption, we have that $$\frac{d^2}{d\beta^2}\mathcal{E}_{\alpha,\beta}[\mu]\le \frac{1}{\beta^2}\mathcal{E}_{\alpha,\beta}[\mu]\le \frac{1}{2\beta^2}\frac{n}{n+1}\left[\frac{1}{\alpha}-\frac{1}{\beta_1}\right].$$
    Since the rightmost term in this inequality is an increasing function of $\beta$ on $[0,\alpha],$ we conclude that
    $$\frac{d^2}{d\beta^2}\mathcal{E}_{\alpha,\beta}[\mu]\le \frac{1}{2\beta^2}\frac{n}{n+1}\left[\frac{1}{\alpha}-\frac{1}{\beta}\right]\le \frac{1}{2\beta_1^2}\frac{n}{n+1}\left[\frac{1}{\alpha}-\frac{1}{\beta_1}\right],$$
    as desired. 
\end{proof} 

We conclude by proving Proposition \ref{prop - sc bound}, which is based on the strong concavity of $E_\alpha.$
\begin{proof}[Proof of Proposition \ref{prop - sc bound}]
    Define $\mathcal{F}:=\left\{\mu\in\mathcal{P}_c(\mathbb{R}^n)\ |\ \mathcal{E}_{\alpha,\beta}[\mu]\le \frac{1}{2}\frac{n}{n+1}\left[\frac{1}{\alpha}-\frac{1}{\beta_1}\right]\right\}.$ Note that $\mathcal{F}$ is non-empty because it contains a regular unit simplex $\nu,$ as 
    $$\mathcal{E}_{\alpha,\beta}[\nu]=\frac{1}{2}\frac{n}{n+1}\left[\frac{1}{\alpha}-\frac{1}{\beta}\right]\le \frac{1}{2}\frac{n}{n+1}\left[\frac{1}{\alpha}-\frac{1}{\beta_1}\right].$$
    Thus, as an immediate consequence of the defintion of $\mathcal{F},$ we have that 
    $E^\mathcal{F}_\alpha(\beta)=E_\alpha(\beta)$
    for all $\beta\in[\beta_0,\beta_1].$ In addition, by Lemma \ref{lem - unif bd}, we see that for each $\mu\in\mathcal{F},$ $\mathcal{E}_{\alpha,\beta}[\mu]$ is strongly concave, with concavity parameter $\frac{1}{2}\frac{n}{n+1}\left[\frac{1}{\alpha}-\frac{1}{\beta_1}\right].$ Thus, by Lemma \ref{lem - inf str cvx}, we conclude that $E_\alpha(\beta)=E_\alpha^\mathcal{F}(\beta)=\inf_{\mu\in\mathcal{F}}\mathcal{E}_{\alpha,\beta}[\mu]$ is strongly concave with the correct concavity parameter. 
\end{proof}

\begin{remark}
    We note that the results of Proposition \ref{prop - sc bound} can be applied to any of the cases described in \eqref{rem - concavity cases}, although if $\alpha=\beta_1,$ then some more work is required to obtain bounds which are better than those described in Lemma \ref{lem - linear interpolation}. In particular, we can (1) fix $\beta_1'\in (\beta_0,\beta_1)$ and use \ref{prop - sc bound} to estimate $E_\alpha(\beta_1')$ from below and (2) use \eqref{eqn - strong convexity estimate} to estimate $E_\alpha(\beta_t)$ in terms of $E_\alpha(\beta_0)$ and $E_\alpha(\beta_1'),$ then replace $E_\alpha(\beta_1')$ with our lower bound. 
\end{remark}

\section{Cascading Phase Transitions}\label{sec-transition}
For the purposes of completeness, we prove the following technical lemma, since we were unable to find a reference:
\begin{lemma}[Compactness of $\mathcal{D}^k$]\label{lem-weak compact}
    Let $\Omega\subseteq\mathbb{R}^n$ be compact. Then for any $k\in\mathbb{N}$, the set $\mathcal{D}^k(\Omega)$ is weakly compact (in duality with $\mathcal{C}^b(\Omega)).$
\end{lemma}
\begin{proof}
    It is not difficult to check that, since $\Omega$ is compact, $\mathcal{P}(\Omega)$ is also compact with respect to the weak topology (see, for example, \cite[Proposition 7.2.1]{Parthasarathy}). Thus, we need only to check that $\mathcal{D}^k(\Omega)$ is weakly closed. To this end, assume that we have a sequence $\{\mu_m\}_m\subseteq \mathcal{D}^k(\Omega)$ such that $\mu_m\rightharpoonup\mu_\infty\in\mathcal{P}(\Omega).$ 
    Assume by way of contradiction that $\mu_\infty\notin \mathcal{D}^k(\Omega),$ so that there are $k+1$ distinct points $x_1,\dots,x_{k+1}$ in $\operatorname{spt}(\mu_\infty),$ and choose $\varepsilon=\frac{1}{3}\min\{|x_i-x_j|\ |\ i\ne j\}.$ 
    Moreover, for $i\in\{1,\dots,k+1\},$ let $\eta_i$ be a (smooth and non-negative) bump function such that $\eta_i(x_i)=1$ and $\operatorname{spt}(\eta_i)\subseteq B_\varepsilon(x_i).$ 
    For any given $i\in\{1,\dots,k+1\},$ the assumption that $x_i\in\operatorname{spt}(\mu_\infty),$ we have that $\int \eta_i\ d\mu_\infty>0.$ Moreover, since $\mu_m\rightharpoonup \mu_\infty,$ we must also have that, for sufficiently large $m,$ $\int \eta_i\ d\mu_m>0$ and hence that $B_\varepsilon(x_i)\cap \operatorname{spt}(\mu_m)\ne\varnothing.$Thus, by choosing $m$ large enough, we have that $B_\varepsilon(x_i)\cap \operatorname{spt}(\mu_m)\ne \varnothing$ for all $i\in\{1,\dots,k+1\}.$ 
    As $\varepsilon$ was chosen to make the balls $B_\varepsilon(x_i)$ disjoint, we deduce that $|\operatorname{spt}(\mu_m)|\ge k$ for $m$ sufficiently large enough, a contradiction to the assumption that $\{\mu_m\}_m\subseteq \mathcal{D}^k(\Omega).$
    Thus, we conclude that $\mu_\infty\in\mathcal{D}^k(\Omega)$ and hence that $\mathcal{D}^k(\Omega)$ is weakly compact. 
\end{proof}

This observation, along with the work on cardinality \cite[Theorem 1.1]{CFP}, and the identification of explicit minimizers in \cite[Theorem 2.1(c)]{DLMSphere} and \cite[Theorem 1]{Frank}, allows us to prove Theorem \ref{thm - card limit}. In other words, we show that, for all interesting values of $\alpha,$ the cardinality of minimizers of $\mathcal{E}_{\alpha,\beta}$ tends to $\infty$ as $\beta\searrow 2.$ 
\begin{proof}[Proof of Theorem \ref{thm - card limit}]
    We first treat the $n\ge 2$ case. As such, fix $\alpha\in (2,4)$ and by way of contradiction, assume that the conclusion of the theorem does not hold. Then there exists an integer $k\in\mathbb{N}$ and a sequence $\beta_m\searrow 2$ such that, for each $m,$ $E^{\mathcal{D}^k(\mathbb{R}^n)}_\alpha(\beta_m)=E_\alpha(\beta_m)$
    so that, by the continuity we proved in Proposition \ref{prop - concave increasing}, 
    $$E_{\alpha}^{\mathcal{D}^k(\mathbb{R}^n)}(2)=\lim_{m\to\infty}E_\alpha^{\mathcal{D}^k(\mathbb{R}^n)}(\beta_m)=\lim_{m\to\infty}E_\alpha(\beta_m)=E_\alpha(2).$$
    In particular, there exists a minimizing sequence $\{\mu_i\}_i\subseteq \mathcal{D}^k(\mathbb{R}^n)$ such that $\lim_{i\to\infty}\mathcal{E}_{\alpha,2}(\mu_i)=E_\alpha^(2).$ As \cite[Proposition 2.1]{DLMSphere} implies that all minimizers of $\mathcal{E}_{\alpha,2}$ have support with diameter at most $e,$ and by the translation invariance of $\mathcal{E}_{\alpha,2},$ we may assume that $\{\mu_i\}_i\subseteq \mathcal{D}^k(\overline{B_{{e}}(0)}),$ which is weakly compact by Lemma \ref{lem-weak compact}. Thus we can find that a subsequence of $\{\mu_i\}_i$ weakly converges to some $\mu_\infty\in\mathcal{D}^k(\overline{B_{{e}}(0)}).$ In particular, since $\mathcal{E}_{\alpha,2}$ is weakly continuous (due, to, for example, \cite[Proposition 7.2]{Santambrogio}), we have that $\mu_\infty\in \mathcal{D}^k(\overline{B_{{e}}(0)})\cap \operatorname{argmin}_{\mathcal{P}(\mathbb{R}^n)}\mathcal{E}_{\alpha,2}[\cdot].$ However, this is a contradiction to \cite[Theorem 2.1(c)]{DLMSphere}, which asserts that $\operatorname{argmin}_{\mathcal{P}(\mathbb{R}^n)}\mathcal{E}_{\alpha,2}[\cdot]$ consists only of `spherical shell' measures, whose supports have infinite cardinality. 

    The proof in the case $n=1$ is nearly identical, with the only distinctions being that we must choose $\alpha\in (2,3)$ and that we instead draw a contradiction to \cite[Theorem 1]{Frank}.
\end{proof}

In light of the results in \cite{CFP}, where the authors show that for all $\alpha>\beta>2,$ any minimizer $\mu$ of $\mathcal{E}_{\alpha,\beta}$ is supported on a set of finite cardinality, Theorem \ref{thm - card limit} implies that, for any fixed $\alpha$ in the range of validity of that theorem, there are an infinite number of phase transitions as $\beta\searrow 2$, and the number of points in the support of any minimizers of $\mathcal{E}_{\alpha,\beta}$ is forced to grow to $\infty.$ 

With the goal of studying weak limits of minimizers of $\mathcal{E}_{\alpha,\beta}$ for $\beta$ near $2,$ we also prove the following technical lemma:

\begin{lemma}[Minimizing Sequences for $\mathcal{E}_{\alpha,\beta}$]\label{lem - minseq}
    Let $\alpha>\beta>0$ and let $\{\beta_m\}_m\subseteq(0,\alpha)$ with $\beta_m\to \beta.$ Moreover, take $\mu_m\in\operatorname{argmin}_{\mathcal{P}(\mathbb{R}^n)}\mathcal{E}_{\alpha,\beta_m}[\cdot].$ Then $\mu_m$ is a minimizing sequence for $\mathcal{E}_{\alpha,\beta}.$ 
\end{lemma}
\begin{proof}
    Pass to the subsequences $\{\beta_{m(\ell)}\}_\ell$ and $\{\mu_{m(\ell)}\}_\ell.$ By assumption, $E_{\alpha}(\beta_{m(\ell)})=E_\alpha^{\{\mu_{m(\ell)}\}_\ell}(\beta_{m(\ell)})$ for any $\ell\in\mathbb{N},$ so by Proposition \ref{prop - concave increasing} and the fact that $\beta_m\to \beta,$ we have that 
    $$E_{\alpha}(\beta)=\lim_{\ell\to\infty}E_{\alpha}(\beta_{m(\ell)})=\lim_{\ell\to\infty}E_\alpha^{\{\mu_{m(\ell)}\}_\ell}(\beta_{m(\ell)})=E_\alpha^{\{\mu_{m(\ell)}\}_\ell}(\beta).$$
    In other words, $\inf_{\ell\in\mathbb{N}}\mathcal{E}_{\alpha,\beta}[\mu_{m(\ell)}]=\inf_{\mu\in\mathcal{P}(\mathbb{R}^n)}\mathcal{E}_{\alpha,\beta}[\mu_{m(\ell)}].$ Thus, there exists a further subsequence of $\{\mu_{m(\ell)}\}$ which is a minimizing sequence for $\mathcal{E}_{\alpha,\beta},$ and hence $\mu_m$ itself is a minimizing sequence for $\mathcal{E}_{\alpha,\beta}.$
\end{proof}

By taking the minimizing sequence in Theorem \ref{thm - card limit} to be of the form addressed in Lemma \ref{lem - minseq}, we deduce the following result which, in a relaxed sense which will be made precise in the statement of the theorem, states that minimizers of $\mathcal{E}_{\alpha,\beta},$ which are known to be discrete by \cite{CFP}, converge weakly to spherical shells as $\beta \searrow 2:$

\begin{theorem}[Weak Convergence to Spherical Shells]\label{thm - weak shells}
     Let $n\ge 2,$  and fix $\alpha\in (2,4).$ Let $\{\beta_m\}_m\subseteq (2,\alpha)$ be such that $\beta_m\searrow 2,$ and define a sequence $\{\mu_m\}_m\subseteq\mathcal{P}(\mathbb{R}^n)$ by $\mu_m\in \operatorname{argmin}_{\mathcal{P}(\mathbb{R}^n)}\mathcal{E}_{\alpha,\beta_m}.$ Then every subsequence of $\{\mu_m\}_m$ has a further subsequence which weakly converges to some spherical shell. 
\end{theorem}
\begin{proof}
    We first address the $n\ge 2$ case. Fix a subsequence $\{\mu_{m(\ell)}\}_\ell$ of $\{\mu_m\}_m.$ As in the proof of Theorem \ref{thm - card limit}, we note that \cite[Proposition 2.1]{DLMSphere} and translation invariance allow us to assume that $\{\mu_{m(\ell)}\}_\ell\subseteq\mathcal{P}(\overline{B_e(0)}).$ By Lemma \ref{lem - minseq}, $\{\mu_{m(\ell)}\}_\ell$ is a minimizing sequence for $\mathcal{E}_{\alpha,2}.$ Hence, since $\mathcal{P}(\overline{B_e(0)})$ is weakly compact and $\mathcal{E}_{\alpha,2}$ is weakly continuous, there exists a further subsequence of $\{\mu_{m(\ell)}\}_\ell$ which converges to $\mu_\infty\in\operatorname{argmin}_{\mathcal{P}(\mathbb{R}^d)}\mathcal{E}_{\alpha,2}.$ Finally, by \cite[Theorem 2.1(c)]{DLMSphere}, $\mu_\infty$ is a spherical shell measure, as desired. 

    The proof of the $n=1$ case is nearly identical, with the only difference being that we use \cite[Theorem 1]{Frank} to conclude that $\mu_\infty$ is of the form described in that theorem. 
\end{proof}

A related result holds in the one-dimensional case, where minimizers of $\mathcal{E}_{\alpha,2}$ for $\alpha\in (2,3)$ are known by \cite[Theorem 1]{Frank}:

\begin{proposition}[Weak Convergence in One Dimension]\label{prop - weak 1d}
    Let $n=1$ and $\alpha\in (2,3).$ Let $\{\beta_m\}_m\in (1,\alpha)$ be such that $\beta_m\to 2,$ and define a sequence $\{\mu_m\}_m\subseteq\mathcal{P}(\mathbb{R}^n)$ by $\mu_m\in \operatorname{argmin}_{\mathcal{P}(\mathbb{R}^n)}\mathcal{E}_{\alpha,\beta_m}.$ Then every subsequence of $\{\mu_m\}_m$ converges weakly to a measure with density as described in \cite[Theorem 1]{Frank}. 
\end{proposition}
\begin{proof}
    For any subsequence of $\{\mu_m\}_m,$ we use the same methods as in the proof of Theorem \ref{thm - weak shells} to recover a further subsequence which tends to $\mu_\infty\in\operatorname{argmin}_{\mathcal{P}(\mathbb{R}^d)}\mathcal{E}_{\alpha,2},$ and such minimizers are totally classified by \cite[Theorem 1]{Frank}.
\end{proof}

A second related result holds for $\alpha>4$ (if $n\ge 2$) and $\alpha>3$ (if $n\ge 1$), which states that as $\beta\nearrow 2,$ minimizers of $\mathcal{E}_{\alpha,\beta},$ whose supports necessarily have positive Hausdorff dimension by \cite[Theorem 1]{BCLR}, converge to unit simplices studied in \cite{DLMSimplex} in the following sense:

\begin{theorem}[Weak Convergence to Unit Simplices]\label{thm - weak simplices}
    For $n=1$ fix $\alpha\ge3$, and for $n\ge 2,$ fix $\alpha>4.$ Let $\{\beta_m\}_m\subseteq (1,2)$ be such that $\beta_m\nearrow 2$ and define a sequence $\{\mu_m\}_m\subseteq\mathcal{P}(\mathbb{R}^n)$ by $\mu_m\in \operatorname{argmin}_{\mathcal{P}(\mathbb{R}^n)}\mathcal{E}_{\alpha,\beta_m}.$ Then every subsequence of $\{\mu_m\}_m$ has a further subsequence which weakly converges to some regular unit simplex.
\end{theorem}

\begin{proof}
    By an argument identical to that the proof of Theorem \ref{thm - weak shells}, every subsequence of $\{\mu_m\}_m$ has a further subsequence converging to $\mu_{\infty}\in\operatorname{argmin}_{\mathcal{P}(\mathbb{R}^d)}\mathcal{E}_{\alpha,2}.$ If $n\ge 2$ and $\alpha>4,$ then by \cite[Corollary 1.4]{DLMSimplex}, $\mu_\infty$ is a unit simplex, as desired. If, instead, $n=1$ and $\alpha\ge 3,$ then \cite[Corollary 2.3(c-d)]{DLMSphere} implies that $\mu_\infty$ is a unit simplex. 
\end{proof}

\begin{remark}
    We note that Theorem \ref{thm - weak shells}, Proposition \ref{prop - weak 1d}, and Theorem \ref{thm - weak simplices} address all cases where global minimizers of $\mathcal{E}_{\alpha,\beta}$ are not explicitly known for $\alpha>\beta\approx 2.$  In particular, if $n\ge 2$, $\alpha\in [2,4],$ and $\frac{-n^2+3n+2}{n+1}\le \beta<2,$ then spherical shells are the only global minimizers of $\mathcal{E}_{4,\beta}$ per \cite[Theorem 1]{FM}. If $(\alpha,\beta)=(4,2)$, then global minimizers of $\mathcal{E}_{4,2}$ are precisely those probability measures which concentrate on appropriate spherical shells and which satisfy the second-moment condition of \cite[Theorem 1]{DLMSimplex}. Finally, $n\ge 2$ and $\alpha>4$, or if $n=1$ and $\alpha\ge 3,$ then for $2<\beta<\alpha,$ regular unit simplex measures are the only global minimizers of $\mathcal{E}_{4,\beta},$ this time by \cite[Theorem 1.2]{DLMSimplex}  and either \cite[Theorem 1.1]{DLMSimplex}(in the $n\ge 2$ case) or \cite[Corollary 2.3(d)]{DLMSphere} (in the $n=1$ case). 
\end{remark}

\section{Bound on Simplex Transition Threshold}\label{sec-threshold}

We conclude by presenting the last contribution in our paper, which concerns the transition threshold function $\alpha_{\Delta^n}(\beta)$ defined in \cite[Theorem 1.5]{DLMSimplex}. Recall that, for $\beta\ge 2,$ $\alpha_{\Delta^n}(\beta)$ is defined as the unique value of $\alpha$ such that $E_{\alpha,\beta}$ is minimized by unit simplices for $\alpha>\alpha_{\Delta^n}(\beta),$ and unit simplices do not minimize $E_{\alpha,\beta}$ for $\alpha<\alpha_{\Delta^n}(\beta).$ While the value of $\alpha_{\Delta^n}(\beta)$ is only explicitly known for $\beta=2,$ the work in \cite[Section 4.1]{DLMSimplex} established an upper bound on $\alpha_{\Delta^n}(\beta)$ which only depends on the dimension through $\min(n,2)$ and \cite[Section 4.2]{DLMSimplex} established a dimensionally-dependent lower bound. In the case where $n\ge 2,$ we provide an alternative lower bound, which has two key advantages. First, the definition of the lower bound is more easily explained than the bound in \cite{DLMSimplex}, which relies on a careful analysis involving the Euler-Lagrange conditions. Second, we are able to show that our lower bound is superior, at least in some cases, and have reason to believe that this is the case in general. To begin, we recall the lower bound from \cite[Section 4]{DLMSimplex}.

\begin{definition}[Lower Bounds and Unimodal Functions from \cite{DLMSimplex}]
    Let the function $f_n:(0,\infty)\to\mathbb{R}$ be defined by 
    $$f_n(t):=\begin{cases}
\frac{2^{-1}-2^{-t}}{t} & \text{ if }n=1\\
\frac{n-(\frac{2n}{n+1})^{t/2}-n(\frac{n-1}{n+1})^{t/2}}{t} & \text{ if }n\ge 2.\end{cases}.$$
For $\beta\ge 2,$ define $\underline{\alpha}_{\Delta^n}(\beta)$ by 
$
\underline{\alpha}_{\Delta^n}(\beta)=\max\{\alpha\ge 2 \ | \  \varphi_n(\alpha)=\varphi_n(\beta)\}.$
\end{definition}
Each $\underline{\alpha}_{\Delta^n}$ was proven to be a lower bound for the threshold $\alpha_\Delta^n$ in \cite[Proposition 4.12]{DLMSimplex}. In particular, those bounds were chosen by choosing points where we felt it likely for unit simplices to violate the Euler-Lagrange conditions for a large range of parameters, and then finding the range of parameters for which there is an Euler-Lagrage violation. Our current approach, on the other hand, is based on identifying a likely family of competitor measures, and checking for which parameters these competitors have lower energy than the simplex. In light of the translation and rotation invariance of this minimization problem, we consider the cross-polytope measures of Definition \ref{def - measures}.
Of course, that definition defines a one-parameter family of measures, so it is useful to find the measure in this family with minimal energy:
\begin{proposition}[Optimal Energy for Cross Polytopes]\label{prop - cross poly energy}
    Let $n\ge 1.$ Then $\mathcal{E}_{\alpha,\beta}[\mu^*_n(r)]$ is minimized when 
    $$r=r_{\alpha,\beta,n}^*=2^{-1/2}\left(\frac{2n-2+2^{\beta/2}}{2n-2+2^{\alpha/2}}\right)^{1/(\alpha-\beta)}.$$
    and
    $$\inf_{r}\mathcal{E}_{\alpha,\beta}[\mu^*_n(r)]=\frac{1}{4n}\left(\frac{(2n-2+2^{\beta/2})^{\alpha/(\alpha-\beta)}}{(2n-2+2^{\alpha/2})^{\beta/(\alpha-\beta)}}\right)\left(\frac1\alpha-\frac1\beta\right).$$
\end{proposition}
\begin{proof}
    For any $r\ge 0,$ we compute 
    \begin{equation}\label{eqn - r energy}
        \mathcal{E}_{\alpha,\beta}[\mu_n^*(r)]=\frac{1}{4n}\left[\left((2n-2)2^{\alpha/2}+2^\alpha\right)\frac{r^\alpha}{\alpha}-\left((2n-2)2^{\beta/2}+2^\beta\right)\frac{r^\beta}{\beta}\right],
    \end{equation}
    so that 
    $$\frac{d}{dr}\mathcal{E}_{\alpha,\beta}[\mu_n^*(r)]=\frac{1}{4n}\left[\left((2n-2)2^{\alpha/2}+2^\alpha\right)r^{\alpha-1}-\left((2n-2)2^{\beta/2}+2^\beta\right)r^{\beta-1}\right].$$
    Solving, we see that $\frac{d}{dr}\mathcal{E}_{\alpha,\beta}[\mu_n^*(r)]=0$ only for $r$ satisfying 
    $$r^{\alpha-\beta}=\frac{(2n-2)2^{\beta/2}+2^\beta}{(2n-2)2^{\alpha/2}+2^\alpha}=2^{(\beta-\alpha)/2}\frac{2n-2+2^{\beta/2}}{2n-2+2{\alpha/2}},$$
    and solving for $r,$ we deduce that 
    $$r=r^*_{\alpha,\beta,n}=2^{-1/2}\left(\frac{2n-2+2^{\beta/2}}{2n-2+2^{\alpha/2}}\right)^{1/(\alpha-\beta)}.$$
    Substituting this in to \eqref{eqn - r energy}, after factoring out $r^\beta$, implies that
    \begin{align*}\mathcal{E}_{\alpha,\beta}[\mu_n^*(r^*_{\alpha,\beta,n})] & =\frac{1}{4n}2^{-\beta/2}\left(\frac{2n-2+2^{\beta/2}}{2n-2+2^{\alpha/2}}\right)^{\beta/(\alpha-\beta)}((2n-2)2^{\beta/2}+2^\beta)\left(\frac{1}{\alpha}-\frac{1}{\beta}\right)\\
    & = \frac{1}{4n}\left(\frac{(2n-2+2^{\beta/2})^{\alpha/(\alpha-\beta)}}{(2n-2+2^{\alpha/2})^{\beta/(\alpha-\beta)}}\right)\left(\frac1\alpha-\frac1\beta\right),\end{align*}
    as desired. 
\end{proof}

We will be interested in studying the relationship between $\mathcal{E}_{\alpha,\beta}[\mu_n^*(r^*_{\alpha,\beta,n})]$ and $\mathcal{E}_{\alpha,\beta}[\mu_n^\Delta].$ In particular, if $\mathcal{E}_{\alpha,\beta}[\mu_n^*(r^*_{\alpha,\beta,n})]<\mathcal{E}_{\alpha,\beta}[\mu_n^\Delta],$ the optimal cross-polytope has lower interaction energy than any unit simplex, so unit simplices cannot globally minimize $\mathcal{E}_{\alpha,\beta},$ so that $\alpha_{\Delta^n}(\beta)\ge \alpha.$ To analyze when this occurs, we will define a family of unimodal functions, analogously to the approach used to find bounds in \cite[Section 4]{DLMSimplex}.
\begin{definition}[A New Family of Unimodal Functions]\label{def - unimodal}
    For $n\ge 1$, define $\varphi_n:(0,\infty)\to (-\infty,0)$ by $$\varphi_n(\gamma):=-\left(\frac{(n+1)(2n-2+2^{\gamma/2})}{2n^2}\right)^{1/\gamma}.$$
\end{definition}
We note that, if $n=1,$ then $f_1\equiv-\sqrt{2}$ which, as will become clear in the following discussion, is a consequence of the fact that the optimal cross-polytope and the unit simplex coincide in the one-dimensional case. We now state a couple of relevant properties of the $\varphi_n.$
\begin{lemma}[Unimodality of $\mathcal{\varphi}_n$]\label{lem - unimodality of phi_n}
    Let $n\ge 2.$ Then $\varphi_n$ is a unimodal function, in the sense that it has a unique global maximum and no other critical points. Moreover, if $t_0\in (1+\frac{n-1}{2n^2},\infty)$ is defined as the unique solution of $$\left(\frac{2n^2}{n+1}t_0-(2n-2)\right)\log\left(\frac{2n^2}{n+1}t_0-(2n-2)\right)=\frac{2n^2}{n+1}t_0\log t_0$$
on $(1+\frac{n-1}{2n^2},\infty)$, then $\varphi_n$ is maximized at $\gamma_0=\frac{2\log\left(\frac{2n^2}{n+1}t_0-(2n-2)\right)}{\log 2}.$
\end{lemma}
\begin{proof}
    Postponed to the appendix. 
\end{proof}

\begin{lemma}[$\varphi_n$ Characterizes Competition between Simplices and Cross Polytopes]\label{lem - opp sign}
    Let $\alpha>\beta>0$, and let $n\ge 1.$ Then $\varphi_n(\alpha)-\varphi_n(\beta)$ and $\mathcal{E}_{\alpha,\beta}[\mu_n^*(r_{\alpha,\beta,n}^*)]-\mathcal{E}_{\alpha,\beta}[\mu_n^\Delta]$ have opposite sign. 
\end{lemma}
\begin{proof}
Applying Proposition \ref{prop - cross poly energy} and using that the energy of the unit simplex is $\frac12\frac{n}{n+1}\left(\frac{1}{\alpha}-\frac{1}{\beta}\right),$ we see that
$$\mathcal{E}_{\alpha,\beta}[\mu_n^*(r_{\alpha,\beta,n}^*)]-\mathcal{E}_{\alpha,\beta}[\mu_n^\Delta]=\frac{1}{4n}\left(\frac{(2n-2+2^{\beta/2})^{\alpha/(\alpha-\beta)}}{(2n-2+2^{\alpha/2})^{\beta/(\alpha-\beta)}}\right)\left(\frac1\alpha-\frac1\beta\right)-\frac12\frac{n}{n+1}\left(\frac{1}{\alpha}-\frac{1}{\beta}\right).$$
Multiplying through by the positive quantity $4n\left(\frac1\beta-\frac1\alpha\right)^{-1},$ we see that this quantity has the same sign as 
$$\frac{2n^2}{n+1}-\frac{(2n-2+2^{\beta/2})^{\alpha/(\alpha-\beta)}}{(2n-2+2^{\alpha/2})^{\beta/(\alpha-\beta)}}.$$
Since the map $t\mapsto t^{\alpha-\beta}$ is injective and increasing on $(0,\infty),$
we may apply this map to both terms in the preceding line to find that 
$\mathcal{E}_{\alpha,\beta}[\mu_n^*(r_{\alpha,\beta,n}^*)]-\mathcal{E}_{\alpha,\beta}[\mu_n^\Delta]$
has the same sign as 
$$\left(\frac{2n^2}{n+1}\right)^{\alpha-\beta}-\frac{(2n-2+2^{\beta/2})^{\alpha}}{(2n-2+2^{\alpha/2})^{\beta}}.$$
Multiplying through by $(2n-2+2^{\alpha/2})^\beta\left(\frac{n+1}{2n^2}\right)^{\alpha}$ and raising both sides to the power of $\frac{1}{\alpha\beta},$ we conclude that $\mathcal{E}_{\alpha,\beta}[\mu_n^*(r_{\alpha,\beta,n}^*)]-\mathcal{E}_{\alpha,\beta}[\mu_n^\Delta]$ has the same sign as 
$$\left(\frac{(n+1)(2n-2+2^{\alpha/2})}{2n^2}\right)^{1/\alpha}-\left(\frac{(n+1)(2n-2+2^{\beta/2})}{2n^2}\right)^{1/\beta}=-(\varphi_n(\alpha)-\varphi_n(\beta)).$$
\end{proof}

\begin{proof}[Proof of Proposition \ref{prop - new threshold bound}]
    Let $\alpha\in(\beta,\underline{\alpha}_{\Delta^n}^*(\beta)).$ By unimodality of $\varphi_n,$ $\varphi_n(\alpha)-\varphi_n(\beta)>0.$ Thus, by Lemma \ref{lem - opp sign}, $\mathcal{E}_{\alpha,\beta}[\mu_n^*(r_{\alpha,\beta,n}^*)]-\mathcal{E}_{\alpha,\beta}[\mu_n^\Delta]<0,$ so that the optimal cross-polytope has strictly lower energy than the unit simplex. Thus, by the definition of $\alpha_{\Delta^n}$ in \cite[Theorem 1.5]{DLMSimplex}, $\alpha<\alpha_{\Delta^n}(\beta),$ as desired. 
\end{proof}

We now turn our attention to providing some non-rigorous evidence which suggests that $\underline{\alpha}_{\Delta^n}^*$ is a better family of bounds than $\underline{\alpha}_{\Delta^n}.$
\begin{example}
    Let $n=2$ and $\beta=2.5.$ Then we can check numerically that $\underline{\alpha}^*_{\Delta^2}(2.5)\approx 3.18$, but $\underline{\alpha}_{\Delta^2}(2.5)\approx 3.07.$
\end{example}

\begin{figure}[H]
    \centering
    \includegraphics[scale=0.25]{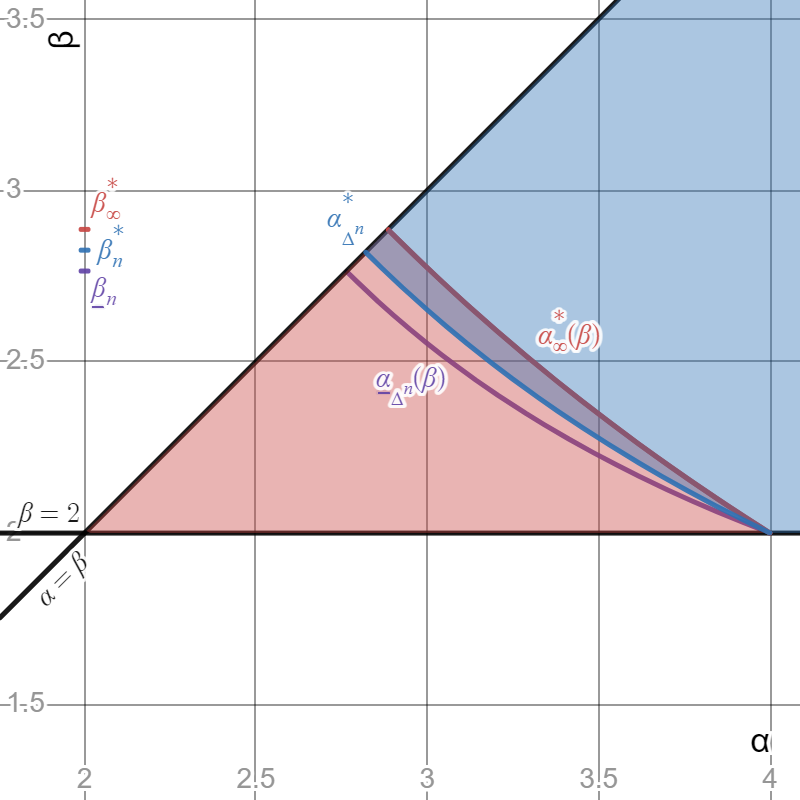}
    \caption{A comparison of $\underline{\alpha}_{\Delta^2}^*$ and $\underline{\alpha}_{\Delta^2}$ appears to indicate that $\underline{\alpha}_{\Delta^2}^*\ge \underline{\alpha}_{\Delta^2}$ for all $\beta> 2$ with $\alpha_{\Delta^2}^*(\beta)>\beta$. This improvement allows us to expand the area in parameter space where unit simplices are known not to minimize $\mathcal{E}_{\alpha,\beta},$ which is represented in this figure as the red region to the left of the graph of $\underline{\alpha}_{\Delta^n}.$ Numerical experiments for higher dimensions $(n\ge 3)$ also showed similar results, which leads us to believe that, for any $n\ge 2,$ $\underline{\alpha}_{\Delta^n}^*\ge \underline{\alpha}_{\Delta^n}$ for all $\beta> 2$ with $\alpha_{\Delta^n}^*(\beta)>\beta.$}
    \label{fig:enter-label}
\end{figure}
These pieces of evidence lead us to the following two conjectures, which we have not been able to prove. We begin with the weaker conjecture, which is also well-supported by numerical evidence first. 
\begin{conjecture}[Superiority of $\underline{\alpha}_{\Delta^n}^*$]
    Let $n\ge 2.$ Then $\underline{\alpha}_{\Delta^n}^*(\beta)\ge \underline{\alpha}_{\Delta^n}(\beta)$ for any $\beta\ge 2,$ with equality if and only if $\beta=2$ or $\alpha_{\Delta^n}^*(\beta)=\beta.$
\end{conjecture}
We conclude with the following stronger conjecture, which states that $\underline{\alpha}_{\Delta^n}^*$ is actually the threshold function:
\begin{conjecture}[Sharpnessof $\underline{\alpha}_{\Delta^n}^*$]
    Let $n\ge 2$ and $\beta\ge 2.$ Then $\underline{\alpha}_{\Delta^n}^*(\beta)=\alpha_{\Delta^n}(\beta).$
\end{conjecture}

\begin{appendices}

\section{Unimodality of $\varphi_n$}\label{secA1}

We now include the postponed proof of Lemma \ref{lem - unimodality of phi_n}. To prove this lemma, it will be useful to recall the following properties of unimodal functions: 
\begin{lemma}[Relevant Properties of Unimodal Functions] \label{lem - uniprop}
Let $g:(a,b)\to (c,d)$ and $h:(c,d)\to\mathbb{R}$ be $C^1.$
\begin{itemize}
    \item If $g$ is strictly increasing, then $h$ is unimodal on $(c,d)$ if and only if $h\circ g$ is unimodal on $(a,b).$ In this case, if $h\circ g$ is maximized at $t_0\in (a,b),$ then $h$ is maximized at $g(t_0)\in (c,d).$
    \item If $h$ is strictly increasing on $(c,d),$ then $g$ is unimodal on $(a,b)$ if and only if $h\circ g$ is unimodal on $(a,b).$ In this case, $g$ and $h\circ g$ are both maximized at the same point $t_0.$
\end{itemize}

\end{lemma}
\begin{proof} Both these facts are immediate consequences of the chain rule. 
    
\end{proof}

\begin{proof}[Proof of Lemma \ref{lem - unimodality of phi_n}] 

 Define a strictly increasing function $\ell:(-\infty, 0)\to (-\infty,\infty)$ by $\ell(s)=-\frac{2\log(-x)}{\log 2}$ and, for each $n\ge 2,$ a strictly increasing function $\xi_n:\left(1+\frac{n-1}{2n^2},\infty\right)\to(0,\infty)$ by 
$$\xi_n(t)=\frac{2\log\left(\frac{2n^2}{n+1}t-(2n-2)\right)}{\log 2}.$$
By applying Lemma \ref{lem - uniprop} twice, $\varphi_n$ is unimodal on $(0,\infty)$ if and only if $\ell\circ \varphi_n\circ\xi_n$ is unimodal on $(1+\frac{n-1}{2n^2},\infty).$ After calculating, we see that  
$$\ell\circ \varphi_n\circ\xi_n(t)=-\frac{\log t}{\log\left(\frac{2n^2}{n+1}t-(2n-2)\right)},$$
so we will only need to check that this function is unimodal on $(1+\frac{n-1}{2n^2},\infty)$ in order to prove the lemma. 

Now, to check unimodality, we compute $(\ell\circ \varphi_n\circ\xi_n)'(t)$. In particular, we have that 
$$(\ell\circ \varphi_n\circ\xi_n)'(t)=-\frac{t^{-1}\log\left(\frac{2n^2}{n+1}t-(2n-2)\right)-\frac{2n^2}{n+1}\left(\frac{2n^2}{n+1}t-(2n-2)\right)^{-1}\log t}{\log^2\left(\frac{2n^2}{n+1}t-(2n-2)\right)}.$$
Multiplying both sides by the positive quantity $t\left(\frac{2n^2}{n+1}t-(2n-2)\right)\log^2\left(\frac{2n^2}{n+1}t-(2n-2)\right),$
we find that 
$(\ell\circ \varphi_n\circ \xi_n)'(t)$ has the same sign as 
$$g_n(t):=-\left(\frac{2n^2}{n+1}t-(2n-2)\right)\log\left(\frac{2n^2}{n+1}t-(2n-2)\right)+\frac{2n^2}{n+1}t\log t$$
on $(1+\frac{n-1}{2n^2},\infty).$ 

As such, to conclude unimodality, it will suffice to show that $g_n$ is concave on $(1+\frac{n-1}{2n^2},\infty)$ with $g_n(1+\frac{n-1}{2n^2})>0.$ For concavity, we compute derivatives. That is, 
\begin{align*}g_n'(t) & =-\frac{2n^2}{n+1}\log\left(\frac{2n^2}{n+1}t-(2n-2)\right)-\frac{2n^2}{n+1}\frac{\frac{2n^2}{n+1}t-(2n-2)}{\frac{2n^2}{n+1}t-(2n-2)}+\frac{2n^2}{n+1}\log t+\frac{2n^2}{n+1}\\
& = -\frac{2n^2}{n+1}\log\left(\frac{2n^2}{n+1}t-(2n-2)\right)+\frac{2n^2}{n+1}\log t.
\end{align*}
and 
\begin{align*}g_n''(t) & =-\left(\frac{2n^2}{n+1}\right)^2\left(\frac{2n^2}{n+1}t-(2n-2)\right)^{-1}+\frac{2n^2}{n+1}t^{-1}\\
& = -\frac{\frac{2n^2}{n+1}(2n-2)}{t\left(\frac{2n^2}{n+1}t-(2n-2)\right)},\end{align*}
which allows us to readily note that $g_n$ is concave on the desired interval, as $\frac{2n^2}{n+1}$, $2n-2,$ $t$, and $\frac{2n^2}{n+1}t-(2n-2)$ are all positive due to our choice of interval and $n$. Finally, we compute
\begin{align*}g_n\left(1+\frac{n-1}{2n^2}\right) &  = (2n-1)\log\left(1+\frac{n-1}{2n^2}\right),\end{align*}
which is positive since $2n-1$ and $\frac{n-1}{2n^2}$ are both positive.

To see where $\varphi_n$ is maximized, we use basic calculus to note that the unique maximum $t_0$ of $\ell\circ \varphi_n\circ\xi_n$ on $(1+\frac{n-1}{2n^2},\infty)$ is defined by the relation $g_n(t)=0,$ or rather,
$$\left(\frac{2n^2}{n+1}t_0-(2n-2)\right)\log\left(\frac{2n^2}{n+1}t_0-(2n-2)\right)=\frac{2n^2}{n+1}t_0\log t_0.$$
Applying Lemma \ref{lem - uniprop} twice, we see that $\varphi_n$ is uniquely maximized at 
$$\xi_n(t_0)=\frac{2\log\left(\frac{2n^2}{n+1}t_0-(2n-2)\right)}{\log 2}.$$

\end{proof}

\end{appendices}

\bibliography{AShanks-ref}

\end{document}